\documentclass[12pt,a4paper]{article}

\usepackage{amssymb, amsmath, amsthm}

\usepackage[cm]{fullpage}
\usepackage[english]{babel}
\usepackage[cp1250]{inputenc}
\usepackage[T1]{fontenc}
\usepackage[pdftex]{graphicx}
\usepackage{booktabs}

\newtheorem{thm}{Theorem}
\newtheorem{lem}{Lemma}
\newtheorem{prop}{Proposition}

\newtheorem{cor}{Collorary}

\title{Compactly supported solution of the time-fractional porous medium equation on the half-line}
\author{\L ukasz P\l ociniczak\thanks{Faculty of Pure and Applied Mathematics, Wroc{\l}aw University of Science and Technology, Wyb. Wyspia{\'n}skiego 27, 50-370 Wroc{\l}aw, Poland},
	Mateusz \'Swita{\l}a\footnotemark[1]$\;^,$\footnote{\underline{Corresponding Author}, e-mail: mateusz.switala@pwr.edu.pl}}
\date{}

\begin{document}
\maketitle

\begin{abstract}
In this work we prove that the time-fractional porous medium equation on the half-line with Dirichlet boundary condition has a unique compactly supported solution. The approach we make is based on a transformation of the fractional integro-differential equation into a nonlinear Volterra integral equation. Then, the shooting method is applied in order to facilitate the analysis of the free-boundary problem. We further show that there exists an exactly one choice of initial conditions for which the solution has a zero which guarantees the no-flux condition. Then, our previous considerations imply the unique solution of the original problem.\\
	
\end{abstract}

\section{Introduction}
In the last few decades an interest in the fractional calculus increased significantly. This can be seen, for instance, in a profound development of theoretical aspects of this branch and its utility as a powerful tool in modelling many physical phenomena. A good example is anomalous diffusion, which has been deeply explored recently \cite{Met00}. Many experiments have been performed and their results indicate either sub- or superdiffusive character of various processes. For instance, apart from its emergence in biomechanical transport \cite{Lev97,Nos83} and condensed matter physics \cite{Ott90}, the anomalous diffusion is also present in percolation of some porous media \cite{El04,Kun01,Ram08}.

For us, the relevant physical experiment is based on a moisture imbibition in certain materials. For some specimens water diffuses at a slower pace than in the classical situation and the usual porous medium equation is not adequate to model this correctly \cite{El04,De06}. To deal with this problem an approach based on the modelling of the waiting times distribution has been proposed and a time-fractional porous medium equation has been introduced to successfully to describe the subdissusive version of the process \cite{Ger10,Pac03,Sun13}. In our previous works \cite{plociniczak13,plociniczak14,plociniczak15} we stated the main assumptions of the model and proved a number of its mathematical properties. 

By $u=u(x,t)$ we denote the (nondimensional) moisture concentration at a point $x\geq 0$ and time $t\geq 0$. We consider the following nonlocal PDE
\begin{equation}
\partial^\alpha_t u = \left(u^m u_x\right)_x, \quad 0<\alpha< 1, \quad m>1,
\label{eqn:DiffEqPDE}
\end{equation}
where the temporal derivative is of the Riemann-Liouville type 
\begin{equation}
\partial^\alpha_t u(x,t) = \frac{1}{\Gamma(1-\alpha)}\frac{\partial}{\partial t}\int_0^t (t-s)^{-\alpha} u(x,s) ds.
\end{equation}
In \cite{plociniczak15} the Reader can find the derivation of the above equation as a consequence of the trapping phenomenon. The initial-boundary conditions are as follows
\begin{equation}
u(x,0) = 0, \quad u(0,t) = 1, \quad x>0, \quad t>0,
\label{eq:CondPDE}
\end{equation}
which models a one-dimensional semi-infinite medium with the interface kept in a contact with water. Since the equation (\ref{eqn:DiffEqPDE}) is degenerate at $u=0$ it is natural to expect that its solution has a compact support for at least some values of $m$ (as in the classical case \cite{Atk71}). This has a straightforward physical meaning, namely the wetting front propagates at a finite speed.

In the previous work \cite{ploswi} we have proved that (\ref{eqn:DiffEqPDE}) possess a unique weak compactly supported solution in the class of sufficiently almost everywhere smooth functions which have a zero at some point. This paper presents a result that weakens those assumptions and leaves only the requirement of physically motivated boundedness. 

\section{Main result}
In this work we will consider the equation of anomalous diffusion in the self-similar form (for more details see \cite{plociniczak14,plociniczak15}). The  transformation leading to it can be commenced by putting $\eta = x t^{-\alpha/2}$ for some $0<\alpha\leq 1$ and denoting $u(x,t)=U(x t^{-\alpha/2})$. In the previous work \cite{ploswi} we proved the existence and uniqueness of the weak compactly supported solution of the resulting transformed problem
\begin{equation}\label{eq:1}
(U^m U')'=\left [(1-\alpha)-\tfrac{\alpha}{2}\eta \tfrac{\,d }{\,d \eta}\right]I^{0,1-\alpha}_{-\tfrac{2}{\alpha}}U(\eta), \quad U(0)=1, \quad U(\infty)=0, \quad m>1,
\end{equation}
under the {\it{a-priori}} assumption that this solution has a zero, i.e. $U(\eta^*)=0$ for some $\eta^* > 0$. Here, the Erd\'elyi-Kober operator $I^{a,b}_c$ (see \cite{kir93,kir97,sneddon}), is defined by
\begin{equation}
I^{a,b}_c U(\eta) = \frac{1}{\Gamma(b)}\int_0^1 (1-s)^{b-1}s^a U(s^\frac{1}{c}\eta)ds.
\label{eqn:EK}
\end{equation}
The main result of this work is to prove that such $\eta^{*}$ exists under the assumption that $U$ is bounded. To ensure the uniqueness of $\eta^{*}$ the auxiliary condition needs to be posed, namely
\begin{equation}
\lim\limits_{\eta\to\eta^-_*}-U(\eta)^m U'(\eta)=0.
\label{co:1}
\end{equation}
Physically, this is simply the no-flux requirement through the wetting front and can be derived from the equation alone (see \cite{ploswi}). 

In what follows we will consider $\alpha$ and $m$ to be a fixed constants. To prove our statement we will consider the following auxiliary problem
\begin{equation}\label{eq:2}
(U^m U')'=\left [(1-\alpha)-\tfrac{\alpha}{2}\eta \tfrac{\,d }{\,d \eta}\right]I^{0,1-\alpha}_{-\tfrac{2}{\alpha}}U(\eta), \quad U(0)=1, \quad U'(0)=-\beta, \quad m>1,
\end{equation}
where $\beta$ is a positive constant. Our technique is based on the fact that the above integro-differential problem can be transformed into an Volterra equation.
\begin{prop}
	If U$(\eta)$ is a continuous solution of 
	\begin{equation}\label{eq:3}
	U(\eta)^{m+1}=1+(m+1)\left[-\beta \eta+\int\limits_{0}^{\eta}((1-\tfrac{\alpha}{2})\eta-z)I^{0,1-\alpha}_{-\tfrac{2}{\alpha}}U(z)\,d z\right].
	\end{equation}
	then it is twice-differentiable and is a solution of \eqref{eq:2}.
\end{prop}
\begin{proof}
	Assume that $U(\eta)$ is a solution of \eqref{eq:3}. The right-hand side of \eqref{eq:3} is $C^2$ since $I^{0,1-\alpha}_{-\frac{2}{\alpha}} U$ is continuous. Then, the left-hand side is also $C^2$ since $m>1$. Differentiating twice the equation \eqref{eq:3} we get \eqref{eq:2}. Thus $U(\eta)$ is also a solution of \eqref{eq:2}.
\end{proof}
Firstly we will prove that the solution of \eqref{eq:3} exists. To do this we need to introduce the function space in which we will operate
\begin{equation}
 X:=\{U\in C[0,\infty]:0\leq U\leq 1\},
\end{equation}
with the uniform norm, i.e. $\left\|U\right\|=\sup_{0\leq t\leq \infty}|U(t)|$. Next, we introduce the function space $M$,
\begin{equation}
  M:=\{U\in C[0,\infty]:U(0)=1,0\leq  U\leq 1\},
\end{equation} 
which is subspace of $X$ and in which the solution of \eqref{eq:3} will be sought. It is easy to see that $X$ is a Banach space ($X$ is subspace of $\mathcal{B}[0,\infty]$, the space of bounded functions, and is closed). For the same reason $M$ is also a Banach space. Moreover, we can make the following simple observation.
\begin{prop}
	\label{prop:M}
	The subspace $M\subset X$ is bounded and convex. 
\end{prop}
\begin{proof}
Let $\gamma\in (0,1)$, $u,v\in M$ and introduce the function $w(x)=\gamma u(x)+(1-\gamma) v(x)$. From definition of $M$ we know that $u(0)=1$, $v(0)=1$ and $0\leq u\leq 1$, $0\leq v\leq 1$. From properties continuous functions the $w$ is also a continuous function. Next, if $u(0)=1$ and $v(0)=1$ then $w(0)=\alpha u(0)+(1-\beta)u(0)=\gamma +(1-\gamma)=1$. And the last properties to show is the boundedness
\begin{equation}
\gamma\cdot 0+(1-\gamma)\cdot 0=0\leq w=\gamma u+(1-\gamma) v\leq \gamma+1-\gamma=1.
\end{equation}
Hence, $w\in M$ which implies that $M$ is convex.
\end{proof}

In order to state the main result we have to construct an appropriate integral operator and show that it possesses a fixed point. First, define the auxiliary operator $S_\beta:M\rightarrow M$ (the well-definiteness will be settled in the following lemma)
\begin{equation}
	S_\beta(Y)=1+(m+1)\left[-\beta \eta+\int\limits_{0}^{\eta}((1-\tfrac{\alpha}{2})\eta-z)I^{0,1-\alpha}_{-\tfrac{2}{\alpha}}Y(z)^{1/(1+m)}\,d z\right],
\end{equation}
then the equation \eqref{eq:3} is equivalent to the fixed-point problem for $S$
\begin{equation}
	S_\beta(Y) = Y,
\end{equation}
what can be seen by the substitution $Y = U^{m+1}$. Now, we can prove several important properties of $S$.
\begin{lem}
\label{lem:S}
Assume that
\begin{equation}
	\beta\geq\frac{2-\alpha}{\sqrt{2\Gamma (2-\alpha)(m+1)}}=:\beta_0,
\end{equation} 
then for $Y\in M$ the following holds. 
\begin{enumerate}
	\item (\textit{Existence of a zero}) There exists $\eta^*(\beta)$ such that $S_\beta(Y)(\eta^*(\beta)) = 0$. Moreover,
	\begin{equation}
		\eta_2(\beta)\leq \eta^*(\beta)\leq \eta_1(\beta),
	\end{equation}
	where 
	\begin{equation}
	\begin{split}
	\eta_1(\beta)&:=\frac{ 4 \beta  \Gamma (2-\alpha )}{(2-\alpha)^2}-\frac{2\sqrt{2} \sqrt{2 \beta ^2 (m+1) \Gamma (2-\alpha )^2-(2-\alpha) \Gamma (3-\alpha)}}{(2-\alpha)^2\sqrt{m+1}}, \\	
	\eta_2(\beta)&:=\frac{2 \sqrt{2} \sqrt{\Gamma (2-\alpha ) \left(\alpha ^2+2 \beta ^2 (m+1) \Gamma (2-\alpha )\right)}}{\alpha ^2\sqrt{m+1}}-\frac{4 \beta  \Gamma (2-\alpha )}{\alpha ^2}.\\
	\end{split}
	\end{equation}
	\item (\textit{Estimates}) We have the following estimates
	\begin{equation}
		g_2(\eta,\beta) \leq S_\beta(Y)(\eta) \leq g_1(\eta,\beta),
	\end{equation}
	where
	\begin{equation}
	\begin{split}
		g_1(\eta,\beta)&:=1+(m+1)\left( -\beta \eta+\frac{(2-\alpha)^2}{\Gamma(2-\alpha)}\frac{\eta^2}{8}\right), \\
		g_2 (\eta,\beta)&:= 1+(m+1)\left( -\beta \eta-\frac{\alpha ^2}{\Gamma(2-\alpha)}\frac{\eta^2}{8}\right).
	\end{split}
	\end{equation}
	\item (\textit{Range}) The operator $S_\beta$ maps $M$ into itself.
\end{enumerate}
\end{lem}
\begin{proof}
	Let us take $Y\in M$ and check when the integral operator $S$ is well-defined. From the definition of the space $M$ we can write an inequality for the Erd\'elyi-Kober operator 
	\begin{equation}
	0\leq{\Large I}^{0,1-\alpha}_{-\tfrac{2}{\alpha}}U(\eta)=\frac{1}{\Gamma(1-\alpha)}\int\limits_0^1(1-s)^{-\alpha}Y(s^{\tfrac{-\alpha}{2}}\eta)^{1/(1+m)}\,d s\leq \frac{1}{\Gamma(2-\alpha)},
	\end{equation}
	which we will use in further parts of the proof. The function $S_\beta(Y)(\eta)$ can be estimated from the above by
	\begin{equation}
	S_\beta(Y)(\eta)\leq 1+(m+1)\left[-\beta \eta+\frac{1}{\Gamma (2-\alpha)}\int\limits_{0}^{\left(1-\frac{\alpha}{2}\right)\eta}\left(\left(1-\frac{\alpha}{2}\right)\eta-z\right)\,d z\right],
	\end{equation}
	what can be integrated to yield $S_\beta(Y)(\eta)\leq g_1 (\eta,\beta)$. On the other hand, in the same way we can find the lower limit of $S_\beta(Y)(\eta)$
	\begin{equation}
	S_\beta(Y)(\eta) \geq 1+(m+1)\left(-\beta \eta -\frac{1}{\Gamma(2-\alpha)}\int\limits_{\left(1-\frac{\alpha}{2}\right)\eta}^{\eta}\left(\left(1-\frac{\alpha}{2}\right)\eta-z\right)\,d z\right),
	\end{equation}
	which implies that $S_\beta(Y)(\eta) \geq g_2 (\eta,\beta)$. 
	
	Now, notice that $g_1$ is decreasing in the interval $[0,\eta^*(\beta)]$ and attains its maximal value equal to $1$ at $\eta=0$. Hence, we get that $S_\beta(Y)(\eta)\leq 1$. Also, from the definition we immediately have $S_\beta(Y)(0)=1$ and thus $S_\beta(Y)\in M$.
	
	We see that both, $g_1$ and $g_2$ are quadratic functions so it is straightforward (but tedious) to find their positivity intervals. We get that $g_1(\eta,\beta) \geq 0 $ for $\eta\leq \eta_1(\beta)$ and the function $g_2(\eta,\beta) \geq 0 $ for $\eta\leq \eta_2(\beta)$. Further, we can see that $\eta_2$ is always greater than zero. The square-root in $\eta_1(\beta)$ is real only when
	\begin{equation}\label{beta}
		\beta\geq\frac{2-\alpha}{\sqrt{2\Gamma (2-\alpha)(m+1)}},
	\end{equation}
	what is satisfied by the assumption. Hence, due to the continuity of $g_1$, $g_2$ and $S_\beta(Y)(\eta)$ we can conclude the existence of at least one $\eta^*(\beta)$ which satisfies inequalities.
	\begin{equation}
		\eta_2(\beta)\leq \eta^*(\beta)\leq \eta_1(\beta).
	\end{equation}
	This concludes the proof.
\end{proof}
We can now see that for the appropriate values of $\beta$ the operator $S_\beta(Y)$ is well-defined on $M$. We can now state the main existence result.

\begin{thm}
	\label{thm:existence}
	For $\beta\geq\beta_0$ the equation \eqref{eq:3} has at least one nonnegative compactly supported solution $Y\in M$. Moreover, $\text{supp} \; Y = [0,\eta^*(\beta)]$. 
\end{thm}
\begin{proof}
It is convenient to introduce yet another operator which maps $M$ into itself 
	\begin{equation}\label{op:1}
	A_\beta(Y)=
	\begin{cases}
	1+(m+1)\left[-\beta \eta+\int\limits_{0}^{\eta}((1-\tfrac{\alpha}{2})\eta-z)I^{0,1-\alpha}_{-\tfrac{2}{\alpha}}Y(z)^{1/(1+m)}\,d z\right], & \text{for}\ \eta\leq \eta^*(\beta), \\
	0, & \text{for}\ \eta> \eta^*(\beta),
	\end{cases}
	\end{equation}
	where $\eta^*$ is the smallest argument that satisfies the equation below (Lemma \ref{lem:S} assures its well-definiteness)
	\begin{equation}
	1+(m+1)\left[-\beta \eta^*(\beta)+\int\limits_{0}^{\eta^*(\beta)}((1-\tfrac{\alpha}{2})\eta^*(\beta)-z)I^{0,1-\alpha}_{-\tfrac{2}{\alpha}}Y(z)^{1/(1+m)}\,d z\right]=0.
	\end{equation}
Our original problem is thus reduced to showing the existence of a solution of the following
\begin{equation}\label{eq:5}
	A_\beta(Y)=Y,
\end{equation}
where operator $A$ is defined above.

From Proposition \ref{prop:M} we have that $M$ is bounded, closed and convex. Moreover, $M\subset X$, i.e. $M$ is a subspace of a Banach space. The kernel of integral operator $A$ is continuous, hence using the Theorem 3.4 from \cite{precup} we conclude that the operator $A$ is completely continuous. Conclusively, we can use Schauder's theorem \cite{precup,zeidler}, which says that equation \eqref{eq:5} has a solution. The form of the compact support comes from the definition of $A_\alpha$. 
\end{proof}	
As for now we know that \eqref{eq:3} has a bounded solution which possesses a zero provided that $\beta$ is large enough. It is also true that inside the interval of the admissible $\beta$ there exists exactly one value if of such for which the no-flux condition \eqref{co:1} is satisfied.

\begin{thm}
	There exists a value of $\beta$ (and hence $\eta^{*}(\beta)$), for which the solution of \eqref{eq:3} satisfies \eqref{co:1}.
\end{thm}
\begin{proof}
	Assume that $U\in M$ and differentiate the equation \eqref{eq:3} to get
	\begin{equation}\label{eq:6}
	U(\eta^*(\beta))^m U'(\eta^*(\beta))=-\beta +\left(1-\frac{\alpha}{2}\right)\int\limits_0^{\eta^*(\beta)}{\Large I}^{0,1-\alpha}_{-\tfrac{2}{\alpha}}U(z)\,d z-\frac{\alpha}{2}\eta^*(\beta){\Large I}^{0,1-\alpha}_{-\tfrac{2}{\alpha}}U(\eta^*(\beta)).
	\end{equation}	
	Since $U(\eta) = 0$ for $\eta\geq \eta^*(\beta)$ we have
	\begin{equation}\label{eq:7}
	\frac{\alpha}{2}\eta^*(\beta){\Large I}^{0,1-\alpha}_{-\tfrac{2}{\alpha}}U(\eta^*(\beta))=0,
	\end{equation}
	which, along with the former formula, implies
	\begin{equation}\label{eq:8}
	U(\eta^*(\beta))^m U'(\eta^*(\beta))=-\beta +\left(1-\frac{\alpha}{2}\right)\int\limits_0^{\eta^*(\beta)}{\Large I}^{0,1-\alpha}_{-\tfrac{2}{\alpha}}U(z)\,d z.
	\end{equation}
	Now, we estimate the magnitude of $U(\eta^*(\beta))^m U'(\eta^*(\beta))$. Using the boundedness of the function $U$ we have
	\begin{equation}
	U(\eta^*(\beta))^m U'(\eta^*(\beta))\leq-\beta+\left(1-\frac{\alpha}{2}\right)\frac{1}{\Gamma(2-\alpha)}\eta_1(\beta)=:f_+(\beta).
	\end{equation}
	On the other hand, we can write
	\begin{equation}
	\begin{split}
	U(\eta^*(\beta))^m U'(\eta^*(\beta))&\geq -\beta+\frac{1-\frac{\alpha}{2}}{\Gamma(2-\alpha)}\int\limits_{0}^{\eta_2(\beta)}\left( 1+(m+1)\left( -\beta z-\frac{\alpha^2}{\Gamma(2-\alpha)}\frac{z^2}{8}\right)\right)^{\tfrac{1}{1+m}}\,d z\\
	& =: f_-(\beta).
	\end{split}
	\end{equation}
	The above estimates can be written as
	\begin{equation}
	 f_-(\beta)\leq U(\eta^*(\beta))^m U'(\eta^*(\beta))\leq f_+(\beta).
	\end{equation}
	Next, we compute the limits $\lim\limits_{\beta\to\infty}f_{\pm}(\beta)$. It is easy to see that
	\begin{gather}
	\lim\limits_{\beta \to \infty}\eta_{1,2}(\beta)=0,
	\end{gather}
	and hence $\lim\limits_{\beta\to\infty}f_{\pm}(\beta)=-\infty$. Moreover,
	\begin{flalign}
	f_+(\beta_0(m))=\frac{\alpha }{\sqrt{2} \sqrt{(m+1) \Gamma (2-\alpha )}}>0, 
	\end{flalign}
	We further see that $\eta_2(0)>0$ so consequently $f_-(0)>0$. Functions
	$f_{\pm}$ are $\beta$-decreasing and continuous. Therefore, since $f_\pm(\beta)$ change their sign when $\beta$ increases from the Darboux Theorem there exists a $\beta^*$ (and hence $\eta^*$) such that the condition $U(\eta^*)^m U'(\eta^*)=0$ is satisfied.	
\end{proof}

Equation (\ref{eq:3}) has thus a unique solution that belongs to $M$, satisfies \eqref{co:1} and has a zero. Now, we can invoke our previous results \cite{ploswi} to conclude that \eqref{eq:1} has a unique bounded solution.  
\begin{cor}[\cite{ploswi}, Corollary 1]
	Let $U$ be a bounded weak solution of \eqref{eq:1} such that $0\leq U(\eta)\leq 1$. Then,
	\begin{itemize}
		\item it is compactly supported with $\text{supp }U = [0,\eta^*]$ for a unique $\eta^*>0$,
		\item it is unique,
		\item it is twice differentiable in the neighbourhood of $\eta$ such that $U(\eta) >0$,
		\item it is monotone decreasing.
	\end{itemize}
\end{cor}

\section{Conclusion}
In this paper we proved that the solution of \eqref{eq:1} has a compact support. We proceeded by reducing our problem to an auxiliary equation to which a shooting method was applied. One of the most important steps was to transform the fractional integro-differential equation into a nonlinear Voltera integral and choose an appropriate function space, where the solution of \eqref{eq:2} was sought. To find a unique $\eta^*$ we posed the no-flux condition \eqref{co:1} which arises from both physical and mathematical reasons. Based on our current results and theorems provided in \cite{ploswi}, we can say that solution of \eqref{eq:1} exist, is unique compactly supported decreasing function. This result rigorously confirms the obvious physical observation and, additionally, gives some useful estimates on the wetting front.

\section*{Acknowledgement}
This research was supported by the National Science Centre, Poland under the project with a signature NCN $2015/17/D/ST1/00625$.


\end{document}